\documentclass[11pt, a4paper]{amsart}

\usepackage{amsmath,amssymb,amsfonts,mathrsfs,amsthm,bbm,enumitem, graphicx, csquotes}

\usepackage{color}

\theoremstyle{plain}
\newtheorem{thm}{Theorem} 
\newtheorem*{thm*}{Theorem} 

\newtheorem{lemma}[thm]{Lemma}
\newtheorem{cor}[thm]{Corollary}

\theoremstyle{definition}
\newtheorem{defn}[thm]{Definition}

\theoremstyle{remark}

\newtheorem{remark}[thm]{Remark}
\newtheorem*{remark*}{Remark}
\newtheorem*{claim*}{Claim}

\newcommand{\tub}[1]{\left\{#1\right\}}

\newcommand{\sqpar}[1]{\left [ #1\right ]}

\newcommand{\tyk}{\mathcal}

\newcommand{\num}[1]{\left | #1\right |}
\newcommand{\para}[1]{\left( #1 \right )}
\newcommand{\flo}[1]{\left\lfloor #1 \right\rfloor}

\renewcommand{\phi}{\varphi}
\renewcommand{\epsilon}{\varepsilon}

\newcommand{\N}{\mathbbm{N}}
\newcommand{\Z}{\mathbbm{Z}}
\newcommand{\Q}{\mathbbm{Q}}
\newcommand{\R}{\mathbbm{R}}
\newcommand{\C}{\mathbbm{C}}

\begin{document}
\title[EVT for Hurwitz Complex CF]{Extreme Value Theory for Hurwitz Complex Continued Fractions}

\author{Maxim Kirsebom}
\address{Maxim Kirsebom, University of Hamburg, Department of Mathematics, Bundesstrasse 55, 20146 Hamburg, Germany}
\email{maxim.kirsebom@uni-hamburg.de}

\thanks{The author is deeply indebted to Seonhee Lim for the many helpful discussions and the much help and assistance in connection with this paper, much of which took place during visits to Seoul National University. Thanks also to Gerardo Gonz\'alez Robert for introducing me to this topic and for helpful discussions. Finally, thanks to the anonymous referees whose suggestions greatly improved the paper}

\begin{abstract}
	The Hurwitz complex continued fraction is a generalization of the nearest integer continued fraction. In this paper we prove various results concerning extremes of the modulus of Hurwitz complex continued fraction digits. This includes a Poisson law and an extreme value law. The results are based on cusp estimates of the invariant measure about which information is still limited. In the process, we get several results concerning extremes of nearest integer continued fractions as well.
\end{abstract}

\maketitle

\section{Introduction}

Continued fraction expansions of real numbers have long been known as an interesting and fruitful venue for applications of extreme value theory (EVT). Interesting because continued fractions digits are almost surely unbounded and have infinite expectation which raises natural questions about the behaviour of their largest digits. Fruitful because continued fraction digits are closely connected to a dynamical system with nice mixing properties and an explicit invariant measure.

We briefly recall some facts about regular continued fractions necessary for further discussion. Every $x\in\R$ may be written as a fraction of the form
\begin{equation}
\label{RCF}
x=a_0(x)+\cfrac{1}{a_1(x)+\cfrac{1}{a_2(x)+\cfrac{1}{\ddots}}}:=[a_0(x);a_1(x),a_2(x),\dots],
\end{equation}
where the (possibly finite) sequence $\tub{a_i(x)}$ consists of natural numbers which we refer to as the regular continued fraction (RCF) digits\footnote{Often the terminology \emph{partial quotients} of $x$ is used for the numbers $a_i$, however, we prefer \emph{digits}. Also, in many instances regular continued fractions are simply referred to as continued fractions. However, since we will discuss multiple different continued fraction algorithms in this paper, we will remain specific throughout.} of $x$. For simplicity we will leave out the dependence on $x$ and write $a_i$ unless the context requires specification. $a_0$ is the integer part of $x$ hence for $x\in [0,1)$ we have $a_0=0$ and only the fraction in \eqref{RCF} remains. Furthermore, the sequence $\tub{a_i(x)}$ is infinite if and only if $x$ is irrational. Since only the infinite case is of interest for extreme value behaviour and since the rational numbers make up a null set anyway, we will focus on numbers in $X:=(0,1)\backslash \Q$.  

The RCF digits of a number $x\in X$ may also be computed using the Gauss map $T:X\to X$ which is defined by
\begin{equation*}
T(x):=\tub{\frac{1}{x}}=\frac{1}{x}-\flo{\frac{1}{x}}.
\end{equation*}
Here $\tub{x}$ and $\flo{x}$ denote the fractional and integer part of $x$ respectively. The $a_i$'s are then given recursively by
\begin{equation*}
a_1:=a_1(x)=\flo{\frac{1}{x}}, \quad a_{n+1}:=a_{n+1}(x)=a_1(T^n(x)).
\end{equation*}
The Gauss map is invariant under the Gauss measure on $X$ which is given by
\begin{equation*}
\mu_G(E)=\frac{1}{\log 2}\int_E \frac{1}{1+x}\, dx
\end{equation*}
where $E\subset X$ is Lebesgue measurable. Let $\tyk{E}$ denote the $\sigma$-algebra of Lebesgue measurable subsets of $X$. As a triple, $(X,T,\mu_G)$ constitutes a dynamical systems with many nice and interesting properties. The first of these is the explicit form of $\mu_G$. For many dynamical system it is possible to prove the existence of an invariant measure, however there are few systems for which this measure is known explicitly. Interestingly, the measure was first stated by Gauss in 1800 \cite{Brezinski} but he gave no explanation for how he discovered it and his rationale remains a mystery to this day. Equipped with the Gauss measure we may think of $X$ as a probability space and the $a_i:X\to\N$ as a sequence of random variables.
The second property of $(X,T,\mu_G)$ which we highlight is a quantitative mixing property with respect to sets belonging to certain sub-$\sigma$-algebras of $\tyk{E}$. We define the property in a general setting.
\begin{defn}[$\psi$-mixing]
	Let $(\Omega,\tyk{B},\mu)$ denote a probability space and let $\xi_j:\Omega\to\R$ denote a stationary sequence of random variables. For natural numbers $u<v$, let $\tyk{M}_{u,v}$ denote the smallest $\sigma$-algebra for which $\xi_u,\dots,\xi_v$ are measurable. Then $\xi_j$ is said to be $\psi$-mixing if for any sets $A\in \tyk{M}_{1,l}$ and $B\in \tyk{M}_{l+n,\infty}$ we have 
	\begin{equation*}
	\num{\mu(A\cap B)-\mu(A)\mu(B)}\leq \psi(n)\mu(A)\mu(B)
	\end{equation*} 
	where $\psi:\N\to \R$ is a function for which $\psi(n)\to 0$ as $n\to\infty$.
\end{defn}
In the case of RCF digits, the random variables $a_{j}=a_1(T^{j-1}(x))$ form a stationary sequence due to the invariance of the Gauss measure with respect to $T$. The $a_j$'s are known to be $\psi$-mixing with respect to the Gauss measure and the function $\psi$ is known to vanish at an exponential rate. This follows from independent work of Kuzmin and Levy in the late 1920's, see for example \cite{Iosifescu2} for the proof and a discussion of its history.

\subsection{Extreme value theory for RCF digits}
Extreme value theory for RCF digits came to life in the 1970's, first through Galambos in 1972. He used the Gauss measure and $\psi$-mixing to prove the following extreme value law.
\begin{thm}[\cite{Galambos}]
	\label{GalambosEVL}
	Let $M_n=M_n(x):=\max_{1\leq i \leq n} a_n(x)$. Then for $r>0$ we have
	\begin{equation}
	\label{RCF_EVL}
	\lim_{n\to\infty}\mu_G\para{M_n\leq \frac{nr}{\log 2}}=e^{-\frac{1}{r}}.
	\end{equation}
\end{thm}
Readers familiar with extreme value theory will recognize this as a Frechet distribution with extremal index equal to 1. Galambos generalized this result in \cite{Galambos2}, showing that $\mu_G$ may be replaced with any measure absolutely continuous with respect to Lebesgue. In 1977, Iosifescu gave the following more general result, proving a Poisson law for general stationary, $\psi$-mixing sequences. 
\begin{thm}[\cite{Iosifescu1}]\label{IosifescuPoisson}
	Let $\xi_j:\Omega\to\R$ denote a stationary and $\psi$-mixing sequence of random variables. For $\omega\in\Omega$ and $v\in\R$, set
	\begin{equation*}
	S_n(\omega,v):=\#\tub{j\in\tub{1,\dots,n}:\xi_j(\omega)>v}.
	\end{equation*}
	Assume that there exists a sequence of functions $u_n:\R\to\R$ such that
	\begin{equation}
	\label{DIG1}
	\lim_{n\to\infty} n\mu(\xi_1(\omega)>u_n(r))=\tau(r)\quad ,\quad r\in\R
	\end{equation}
	for some real-valued function $\tau$. Then, for all $r\in\R$ and any $j\in\N$,
	\begin{equation} 
	\lim_{n\to\infty} \mu\tub{S_n(\omega,u_n(r))=j}=e^{-\tau(r)}\frac{\tau(r)^j}{j!}.
	\end{equation}
\end{thm}
It was strightforward for Iosifescu to apply this theorem to RCF digits (\cite[Theorem 2]{Iosifescu1}) since $\psi$-mixing for RCF digits was known and property \eqref{DIG1} follows easily from the form of the Gauss measure. Similar to Galambos, he gave an argument for why the Poisson law holds for any measure which is absolutely continuous with respect to Lebesgue. Note that one obtains Galambos' theorem by setting $j=0$ in Iosifescu's result for RCF digits. 
It is interesting to note that Iosifescu's result for RCF digits appeared as early as 1940 in a paper by Doeblin \cite{Doeblin}. However, as Iosifescu explains in \cite{Iosifescu1}, Doeblin's proof contains a mistake. Galambos was unaware of Doeblin's result when proving his extreme value law, and Iosifescu obtained his proof by applying ideas from Galambos paper to fix Doeblin's mistake.

The era offered more results concerning the maximum of RCF digits. Galambos proved an iterated logarithm type theorem for the maximal RCF digits \cite{Galambos3}, which was improved by Philipp \cite{Philipp1} to give a complete answer to a conjecture of Erd\H{o}s. Also in \cite{Philipp1}, an upper bound on the rate of convergence in \eqref{RCF_EVL} was provided. Diamond and Vaaler \cite{DiamondVaaler} showed that the partial maximum is responsible for the failure of the law of large numbers for RCF digits.

The topic of extreme value statistics for RCF digits has gained interest again in recent years. Philipp's rate of convergence was improved by Ghosh, Kirsebom and Roy \cite{GhoshKirsebomRoy}, while a refinement of Iosifescu's theorem was presented by Zweim\"uller in \cite{Zweimuller}. Of a slightly different flavour, Chang and Chen \cite{ChangChen} investigated the Hausdorff dimension of certain sets defined via the largest RCF digits.

\subsection{Extreme value theory for other CF algorithms}

The regular continued fractions considered up to this point are by far not the only continued fractions in existence. For RCF digits we have a good understanding of the statistical behaviour of largest digits. For other CF algorithms however, results of this nature are scarce and many interesting questions remain open. To our knowledge, the only works on extreme value theory for other algorithms are the papers by Chang and Ma \cite{ChangMa} which considers the case of Oppenheim CF, by Shen, Xu and Jing \cite{ShenXuJing} which studies the case of continued fraction defined over the field of formal Laurent series and by Nakada and Natsui \cite{NakadaNatsui} which investigates fibred systems. In a different but related direction, González Robert \cite{Robert1} recently proved a Borel-Bernstein Theorem for Hurwitz complex continued fractions. The methods used in \cite{ChangMa} are somewhat different to the ones described for RCF's since the invariant measure associated to Oppenheim CF is infinite. Instead, the metric theory is developed for the Lebesgue measure with respect to which the associated dynamical system is not $\psi$-mixing. Also \cite{ShenXuJing} is somewhat different in that an iterated logarithm type result is proven as opposed to a distributional result.  

To the contrary, \cite{NakadaNatsui} takes the same approach as described above and applies it to fibred systems. Many CF algorithms, including several complex CF algorithms, satisfy the conditions for being fibred systems, see \cite{Schweiger1} for some examples. 

Under certain assumptions, fibred systems were proven by Waterman \cite{Waterman} to admit an invariant measure absolutely continuous with respect to Lebesgue and under further assumptions Schweiger \cite{Schweiger1}, \cite{Schweiger2} showed that the fibred system is $\psi$-mixing with respect to this invariant measure. Nakada and Natsui use this fact to formulate general sufficient conditions on the invariant measure in order to get analogues of Theorem \ref{GalambosEVL} as well as the results in \cite{DiamondVaaler} and \cite{Philipp1}. 

They further proved that these various assumptions are satisfied for the Jacobi-Perron multidimensional CF algorithm, but for many other CF algorithms this is not known.

The main aim of this paper is to develop extreme value theory for complex continued fractions, more specifically the variant introduced by Hurwitz \cite{Hurwitz}. In the process we \enquote{pick up} analogue results for nearest integer continued fractions which we present first before continuing to the complex continued fractions and our main results.

\subsection{Extreme value theory for nearest integer continued fractions}

This subsection serves two purposes. First, it allows us to state some new results concerning extreme value theory for nearest integer continued fractions (NICF). These results are new in the sense that they appear not to have been stated elsewhere before. However, aside from a small calculation, the proofs simply combine results proven elsewhere, hence in that sense the novelty is limited. Second, the complex continued fractions which are the main focus of this paper are a generalization of NICF for real numbers. Hence this subsection serves as a stepping stone between the RCF's and the Hurwitz complex continued fractions. 

NICF's work similar to RCF's, the main difference being that we round to the nearest integer. We consider the fundamental domain $X=\left [-\frac12,\frac12\right)$ and for an $x\in X$ we take its inverse $\frac{1}{x}$ and subtract the integer nearest to it. This brings us back to $X$ where we repeat the process. This enables us to write $x$ as
\begin{equation}
\label{NICF_1}
x=\cfrac{1}{\sqpar{\frac{1}{x}}_{\textup{N}}+\cfrac{1}{\sqpar{\frac{1}{\frac{1}{x}-\sqpar{\frac{1}{x}}_{\textup{N}}}}_{\textup{N}}+\cfrac{1}{\ddots}}},
\end{equation}
where $\sqpar{\cdot}_N$ denotes the nearest integer function. Similar to the RCF, the NICF digits may be generated through a dynamical system. However, the layout of the digits will vary slightly from \eqref{NICF_1}, the benefit being that our transformation gets a simpler expression. We define the transformation $T_{\text{N}}:X\to X$ by
\begin{equation*}
T_{\text{N}}(x)=\begin{cases}
\frac{\epsilon}{x}-\flo{\frac{\epsilon}{x}+\frac12} & , \enskip x\neq 0\\
0 & , \enskip x=0.
\end{cases}
\end{equation*}
where $\epsilon$ denotes the sign of $x$. Using this definition the modulus of the nearest integers and their signs are recorded in separate digits. The NICF expansion of $x\in X$ then becomes 
\begin{equation}
\label{NICF_2}
x=\cfrac{\epsilon_1}{b_1(x)+\cfrac{\epsilon_2}{b_2(x)+\cfrac{\epsilon_3}{\ddots}}},
\end{equation}
where 
\begin{equation*}
b_1(x)=\flo{\num{\frac{1}{x}+\frac12}}, \quad b_n(x)=b_1(T_{\text{N}}^{n-1}(x))=\flo{\num{\frac{1}{T_{\text{N}}^{n-1}(x)}+\frac12}}
\end{equation*}
and
\begin{equation*}
\epsilon_n=\textup{sgn}\para{T_{\text{N}}^{n-1}(x)}.
\end{equation*}
In this notation the $\epsilon_n$'s record whether the nearest integer digits from \eqref{NICF_1} change sign. Indeed the NICF digits in \eqref{NICF_1} may be written as 
\begin{equation*}
	\epsilon_1b_1(x), (\epsilon_1\epsilon_2)b_2(x),\dots,(\epsilon_1\dots \epsilon_n)b_n(x),\dots
\end{equation*}

The map $T_{\text{N}}$ admits an invariant measure $\mu_N$ on $\left [-\frac12,\frac12\right)$ whose density is given by
\begin{equation*}
\rho(x)=\begin{cases}
\frac{1}{\log(G)}\frac{1}{G+x},&\quad \text{for } x\in \left[0,\frac12\right )\\
\frac{1}{\log(G)}\frac{1}{G+1+x},&\quad \text{for } x\in \left[-\frac12, 0\right )
\end{cases}
\end{equation*}
where $G=\frac{\sqrt{5}+1}{2}$ is the golden ratio. It is also known to be $\psi$-mixing with respect to this measure. Indeed, the function $\psi(n)$ is known to decay faster than $O\para{\theta^n}$ where $\theta=\frac{3}{4}$. Both of these results were proven by Rieger (\cite{Rieger1,Rieger2}).

Consider now $b_n:X\to\N$ as a sequence or random variables. By $\mu$-invariance $b_n$ forms a stationary sequence. Define
\begin{equation*}
S_n(x,r)=\#\tub{j\in\tub{1,\dots,n}:b_j(x)>r}.
\end{equation*}
Now, the only property missing in order to apply Theorem \ref{IosifescuPoisson} to $S_n(x,r)$ is the property \eqref{DIG1}. However, this follows from the ensuing calculation. We shorten notation by writing $\tub{b_n>j}$ instead of $\tub{x\in X:b_n(x)>j}$. Also, we use the notation $\sim$ to indicate that the related quantities have the same limit. We obtain
\begin{align*}
\mu_N\tub{b_n>j}&=\mu_N\tub{b_1>j}\\
&=\mu_N\tub{\flo{\num{\frac{1}{x}+\frac12}}>j}\\
&\sim\mu_N\tub{\num{\frac{1}{x}}>j}\\
&=\frac{1}{\log G}\para{\int_{-\frac1j}^0 \frac{1}{G+1+x}\,dx+\int_{0}^{\frac1j} \frac{1}{G+x}\,dx}\\
&=\frac{1}{\log G}\para{\log\para{1+\frac{1}{Gj}}-\log\para{1+\frac{1}{(G+1)j}}}\\
&=\frac{1}{\log G}\para{\frac{1}{G}+\frac{1}{G+1}}\frac{1}{j}+O\para{\frac{1}{j^2}}\\
&=\frac{1}{\log G}\frac{1}{j}+O\para{\frac{1}{j^2}}.
\end{align*}
In the above, the fact that we did not change the limit by ignoring the floor function and the added $\frac12$ requires a small calculation but is intuitively clear. Futhermore, we made use of the power series expansion of $\log(1+x)$. From this we see that
\begin{equation}\label{cuspestiNICF}
\lim_{n\to\infty}	n\mu_N\tub{b_n>\frac{nr}{\log G}}=\frac{1}{r}
\end{equation}
which verifies \eqref{DIG1} for $u_n(r)=\frac{nr}{\log G}$. Hence we have proven the Poisson law of exceedances for the NICF digits.
\begin{thm}\label{thm:PoissonNICF}
	Let $u_n(r)=\frac{nr}{\log G}$. For all $r>0$ and any $j\in\N$ 
	\begin{equation}
	\lim_{n\to\infty} \mu_N \tub{x\in X:S_n(x,u_n(r))=j}= e^{-\para{\frac{1}{r}}} \frac{r^{-j}}{j!}.
	\end{equation}
\end{thm}
This leads to the immediate corollaries
\begin{cor}\label{cor:NICF}
	Let $M_n(x):=\max_{1\leq i \leq n} \num{b_i(x)}$ and let $M_n^{(k)}(x)$ denote the $k$'th largest element among $\tub{b_1(x),\dots, b_n(x)}$. 
	
	For all $r>0$ 
	\begin{equation*}
	\lim_{n\to\infty} \mu_N \tub{x\in X: M_n^{(k)}\leq \frac{nr}{\log G}}=e^{-\frac{1}{r}}\sum_{j=0}^{k-1} \frac{r^{-j}}{j!}.
	\end{equation*}
	
	In particular,
	\begin{equation*}
	\lim_{n\to\infty} \mu_N \tub{x\in X: M_n(x)\leq \frac{nr}{\log G}}=e^{-\frac{1}{r}}.
	\end{equation*}
\end{cor}
The last result is a direct analogue of Galambos' Theorem for RCF, the difference appearing only in the constant used to normalize the maximum.
\begin{remark}
	When comparing the RCF with the NICF of certain real numbers, it appears plausible that the behaviour of their largest digits should be similar. Taking $\pi$ as an example we have the two different expansions\footnote{See https://oeis.org/A001203 and https://oeis.org/A133593 for more digits of either expansion.}
	\begin{align*}
	\pi_{RCF}&=[3; 7, 15, 1, 292, 1, 1, 1, 2, 1, 3, 1, 14, 2, 1, 1, 2, 2, 2, 2, 1, 84, 2, 1, 1,\dots]\\
	\pi_{NICF}&=[3; 7, 16, 294, 3, 4, 5, 15, 3, 2, 2, 2, 2, 3, 85, 3, 2, 15, 3, 14, 5, 2, 6, 6,\dots].
	\end{align*} 
	One observes a certain similarity in the larger digits, say 292 and 294 as well as 84 and 85. A big difference is the lack of 1's in the NICF expansion leading to shorter intervals between the large digits. Heuristically this indicates that the partial maximum for the NICF digits should grow faster than the partial maximum for RCF digits. This is indeed reflected in the extreme value laws for the two expansions. In the NICF case the normalizing sequence $nr$ must be multiplied with the larger constant $(\log G)^{-1}>(\log 2)^{-1}$ in order to obtain the same distribution as for the RCF.  
\end{remark}

We also obtain a rate of convergence for the limits in Theorem \ref{thm:PoissonNICF} and Corollary \ref{cor:NICF}. Let $l_n$ be a sequence which satisfies
\begin{equation*}
l_n \theta^{l_n}=n.
\end{equation*}
where again $\theta=\frac{3}{4}$. Note that in particular $l_n=o(\log n)$.
\begin{thm}\label{thm:conrateNICF}
	For all $k\in \N$ and all $r>0$, the rate of convergence in \eqref{ComplexPoisson}, \eqref{Mklimit} and \eqref{eq:ComplexGalambos} is bounded by
	\begin{equation*}
	O\para{\frac{1}{\min\para{r,r^2}}\frac{l_n}{n}}.
	\end{equation*}
\end{thm}
\begin{proof}
	The proof follows exactly the proof of Theorem 1.1 of \cite{GhoshKirsebomRoy}. Only in a few places, small adjustments have to be made due to the different density of the invariant measure. The adjustments are done using \eqref{cuspestiNICF} and the derivations leading up to it, in particular the power series expansion of $\log(1+x)$. Due to the simplicity of the adaptations we leave the details to the reader.  
\end{proof}

\section{Complex continued fractions and main results}

While there are many CF algorithms for real numbers, the RCF algorithm has a strong sense of being the most \enquote{natural}. This is due to its simplicity, strong properties and its many connections to other fields like analytic number theory, dynamical systems and hyperbolic geometry. In the complex realm no CF algorithm reigns supreme in the same way. Part of the explanation is that the RCF algorithm for real numbers does not generalize in a meaningful way to the complex numbers\footnote{The argument that follows was kindly presented to me by Gerardo Gonz\'alez Robert.}. A naive approach to generalizing the RCF would see us using complex inversion on the square
\begin{equation*}
S:=\tub{z=x+iy\,:\,0<x,y\leq 1}
\end{equation*}
and applying the floor function in both dimensions. More precisely, define the transformation $T:S\to S$ by
\begin{equation*}
T(z)=\frac{1}{z}-\flo{\frac{1}{z}}
\end{equation*}
where $\flo{\cdot}$ denotes the complex floor function given by $\flo{z}=\flo{x+iy}=\flo{x}+i\flo{y}$. Set $a_1(z):=\flo{\frac{1}{z}}$ and $a_{n+1}(z)=a_1(T^n(z))$. Using these definitions, uncountably many different elements of $S$ would be assigned the same sequence of digits $\tub{a_n}$. A specific example is the region bounded by the circles $C_1:=\tub{z\in\C:\num{z-\frac12}=\frac12}$ and $C_2:=\tub{z\in\C:\num{z-(\frac12+i)}=\frac12}$ as well as the line segment connecting $1$ and $1+i$. This region gets mapped to itself under $T$ and in the process every element $z$ in this region gets assigned $a_n(z)=-i$ for all $n\in \N$. Clearly this does not lead to a useful continued fractions representation of numbers in $S$.

Many alternative approaches to complex continued fractions exist, see \cite{LukyanenkoVandehey} for an overview of some of them and their properties and references. The approach which we study in this paper is a generalization of the NICF algorithm for real numbers.

\subsection{Hurwitz complex continued fractions}

Denote by $\Z[i]=\tub{x+iy:x,y\in\Z}$ the Gaussian integers and denote by $[z]_i$ the Gaussian integer nearest to $z$. We apply the convention that ties are broken by rounding down in both real and complex part, for example $[z]_i = a + bi$ if $z=(a+0.5)+ (b+ 0.5)i$ for $a,b \in \Z$. This ensures that $[\cdot]_i$ is well-defined, however, the choice of rounding will play no role for our results since the convention relates only to a set of measure zero. 

Let $B=\tub{z=x+iy\in\C : -\frac12\leq x,y<\frac12}$. Analogous to the Gauss map, we define the Hurwitz map $T:B\to B$ by
\begin{align}
Tz=\frac{1}{z}-\sqpar{\frac{1}{z}}_i
\end{align}
For a given $z \in \C$, consider the sequence $z_k$ given by $z_0=z-[z]_i$ and 
\begin{align}
z_k=Tz_{k-1}=\frac{1}{z_{k-1}}-\sqpar{\frac{1}{z_{k-1}}}_i, \;\; \mathrm{i.e.} \;\;\; z_{k-1}=\frac{1}{\sqpar{\frac{1}{z_{k-1}}}_i+z_k},
\end{align}
for $k\geq 1$.
By setting $a_0=a_0(z):=[z]_i$ and $a_k=a_k(z):=\sqpar{\frac{1}{T^{k-1}z_0}}_i$ for $k\geq 1$, we get the Hurwitz complex continued fractions (HCCF) expansion of $z$ written as
\begin{align}
z=a_0+\cfrac{1}{a_1+\cfrac{1}{a_2+\cfrac{1}{\ddots}}}:=[a_0;a_1,a_2,\dots].
\end{align}
It is well known that this expansion converges and provides a meaningful representation of complex numbers, see \cite{Hensley} for an introduction to the HCCF expansion. 
Let $\tyk{B}$ denote the $\sigma$-algebra of Borel subsets of $B$ and let $\lambda$ be the Lebesgue measure on $B$. It is known that there exists a unique measure on $B$ which is $T$-invariant and absolutely continuous with respect to Lebesgue (see \cite{Hensley}, Section 5.7). We denote this measure by $\mu$. Thus we have a dynamical system $(T, B, \mu)$ where $T:B\to B$ is a $\mu$-preserving map. Consider the sequence of real-valued random variables $\num{a_i}:B\to\R_+$. 
As previously indicated we are particularly interested in the occurrence of large values of $\num{a_i(z)}$. 

We are finally ready to state the main results of this article. The first is a Poisson law for HCCF, analogue to Iosifescus result applied to RCF's.
\begin{thm}\label{thm:main1}
	For $z\in B$ and $v\in\R_+$, let
	\begin{equation*}
	S_n(z,v):=\#\tub{i\in\tub{1,\dots,n}:\num{a_i(z)}>v}.
	\end{equation*}
	There exist $C>0$ such that for all $r>0$ and any $j\in\N$ 
	\begin{equation}\label{ComplexPoisson}
	\lim_{n\to\infty} \mu \tub{z\in B:S_n(z,Cr\sqrt{n})=j}= e^{-\para{\frac{1}{r^2}}} \frac{r^{-2j}}{j!}.
	\end{equation}
\end{thm}
This leads to the immediate corollaries
\begin{cor}
	\label{cor:main1}
	Let $M_n(z):=\max_{1\leq i \leq n} \num{a_i(z)}$ and let $M_n^{(k)}(z)$ denote the $k$'th largest element among $\tub{\num{a_1(z)},\dots, \num{a_n(z)}}$. There exists $C>0$ such that for all $r>0$ 
	\begin{equation}\label{Mklimit}
	\lim_{n\to\infty} \mu \tub{z \in B: M_n^{(k)}\leq Cr\sqrt{n}}=e^{-\frac{1}{r^2}}\sum_{j=0}^{k-1} \frac{r^{-2j}}{j!}.
	\end{equation}
	
	In particular,
	\begin{equation}\label{eq:ComplexGalambos}
	\lim_{n\to\infty} \mu \tub{z \in B: M_n(z)\leq Cr\sqrt{n}}=e^{-\frac{1}{r^2}}.
	\end{equation}
\end{cor}
Note that \eqref{eq:ComplexGalambos} provides an analogue of Galambos's extreme value law for HCCF's. The corollary is obtained simply by setting $j=0$ in Theorem \ref{thm:main1}.

\section{The invariant measure with respect to the Hurwitz map}

The strategy of proof of Theorem \ref{thm:main1} is clear. If we can satisfy the conditions of Theorem \ref{IosifescuPoisson} the result follows. 
Fortunately, the $\psi$-mixing property of $T$ is known and it is even known to be mixing at an exponetial rate, i.e. $\psi(n)=O(\theta^n)$ for some $\theta<1$. This exponential $\psi$-mixing result was first stated by Nakada in \cite[Corollary 2]{Nakada}, but the proof relied on work by Schweiger \cite{Schweiger5} which turned out to contain a serious gap. As explained in \cite{EiItoNakadaNatsui}, Schweiger later was able to recover the result with a sub-exponential rate in \cite{Schweiger3} and finally the exponential rate in \cite{Schweiger4}, meaning that Nakada's result is indeed correct. 
Note also that the sequence $\num{a_i}$ forms a stationary sequence, this is an easy consequence of $\mu$ being $T$-invariant. 

The greater challenge is to satisfy \eqref{DIG1} in Theorem \ref{IosifescuPoisson}. For this we need more specific information about the unique $T$-invariant measure absolutely continuous with respect to Lebesgue. We first cite the following theorem by Hensley which is central to our proof. 
%
\begin{lemma}[\cite{Hensley}, Theorem 5.5]\label{lem:1}
	The density function $\rho$ of the measure $\mu$   is continuous except possibly along the intersections of the circles $|z\pm 1|=1, |z\pm i|=1$ and $|z\pm 1 \pm i|=1$ with $B$. It is real analytic on each of the 12 open regions into which the interior of $B$ is dissected by these circles. 
	Moreover, the measure $\mu$ is symmetric under complex conjugation and multiplication by $i$.
\end{lemma}
The 8 arcs and 12 regions in $B$ are denoted by $\gamma_i$ and $A_i$ respectively. See Figure \ref{Fig1}.
\begin{figure}[h]
	\centering
	\includegraphics[width=7.5cm]{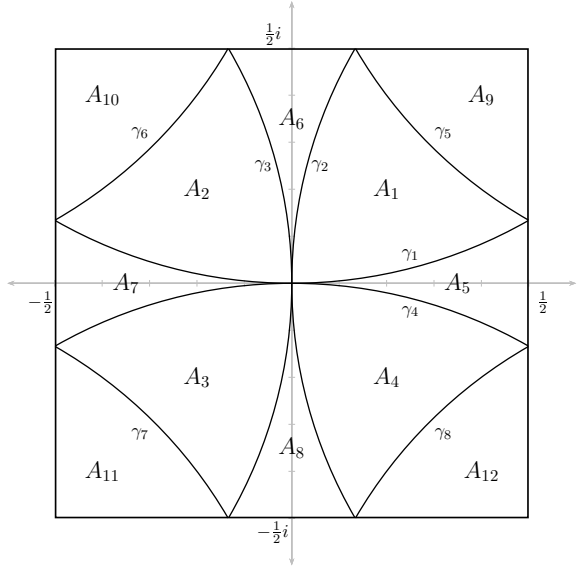}
	\caption{Dissection of B.}
	\label{Fig1}
\end{figure}
An important further fact about the density $\rho$ is that the limits 
\begin{equation*}
C_i : = \lim_{A_i\ni z \to 0} \rho(z),\quad\text{for}\quad i=1,\dots,8
\end{equation*}
exist and are strictly positive, i.e $C_i>0$. This can be seen with the help of various facts from recent works of Hiary and Vandehey \cite{HiaryVandehey} as well as Ei, et al \cite{EiItoNakadaNatsui}. As explained in \cite{HiaryVandehey}, the density $\rho$ may be expressed as
\begin{equation}
\label{eq:DensityForm}
\rho(z)=\int_{V_i} \frac{1}{\num{zw+1}^4}\,dw,\quad\quad\text{if } z\in A_i 
\end{equation}
where $V_i$ is a certain set related to $A_i$ whose exact definition we will not get into here. The sets $V_i$ are not completely understood at this point (if they were, we would understand the measure $\mu$ much better), but they were thoroughly studied in \cite{EiItoNakadaNatsui} where it was shown, among other things, that they are measurable and of positive Lebesgue measure. Hence we see that from \eqref{eq:DensityForm} that the limit for $z\to 0$ inside either of the regions $A_i$, $i=1,\dots,8$, exists and equals the Lebesgue measure of the corresponding set $V_i$. Due to the symmetries mentioned in Theorem \ref{lem:1} the limits inside $A_1$, $A_2$, $A_3$ and $A_4$ must be identical and the same applies to the limits inside $A_5$, $A_6$, $A_7$ and $A_9$. Hence we may set:
\begin{gather}
\begin{split}
\label{ConstantNotation}
\tilde{C}:=C_i,\quad \text{for } i=1,2,3,4.\\
C':=C_i,\quad \text{for } i=5,6,7,8.
\end{split}
\end{gather}
Nakada had much earlier showed \cite[Theorem 2]{Nakada} that the measure $\mu$ is equivalent to the Lebesgue measure, i.e. there exists a constant $Q>0$ such that for all $E\in\tyk{B}$ 
\begin{equation*} 
\frac{1}{Q}\lambda(E)\leq \mu (E) \leq Q\lambda(E).
\end{equation*}

\subsection*{Proof of main results}

The main technical lemma of this paper is the following.
\begin{lemma}\label{lem:integral}
	For some constant $H>0$ we have
	\begin{equation*}
	\mu\tub{z\in B:|a_1(z)| > j} =\frac{H}{j^2}+O\para{\frac{1}{j^3}},\quad \text{ as } j\to\infty.
	\end{equation*}
\end{lemma}
\begin{proof}
	For the purpose of this proof we occasionally shorten notation by writing $\tub{\num{a_1}>j}$ instead of $\tub{z\in B: \num{a_1(z)}>j}$. 
	
	Set
	\begin{equation*}
	A:=\tub{\num{a_1}>j}=\tub{\num{\sqpar{\frac{1}{z}}_i}>j},
	\end{equation*}
	hence our goal is to determine the asymptotics of $\mu(A)$. 
	Let $\phi: \C\backslash\tub{0}\to\C\backslash\tub{0}$ be given by $\phi(z)=\frac1z$ and notice that $\phi$ is a bijection which maps disks of radius $r>0$ to complements of disks of radius $\frac1r$.
	Then for $A'=\tub{z\in\C:\num{\sqpar{z}_i}\geq j}$,
	\begin{equation*}
	A=\phi( A').
	\end{equation*}
	We define two sets $\underline{A'}$ and $\overline{A'}$ by
	\begin{equation*}
	\underline{A'}=\tub{\num{z}\geq j+\frac{1}{\sqrt{2}}}\enskip\text{and}\enskip
	\overline{A'}=\tub{\num{z}\geq j-\frac{1}{\sqrt{2}}}.
	\end{equation*}
	A simple geometric argument (see Figure \ref{Fig2}) shows that
	\begin{equation}
	\label{eq:setinclusion}
	\underline{A'}\subset A'\subset \overline{A'}
	\end{equation}
	Figure \ref{Fig2} shows $A'$, $\underline{A'}$ and $\overline{A'}$ for some unspecified $j>2$ as well as the set $B$. $A'$ is the (unbounded) grey region, $\underline{A'}$ is the region \emph{outside} the green circle while $\overline{A'}$ is the region \emph{outside} the red circle. The blue dots are the Gaussian integers. The set $A$, not drawn in the figure, is some neighborhood of $0$ inside $B$. 
	\begin{figure}[h]
		\centering
		\includegraphics[width=8cm]{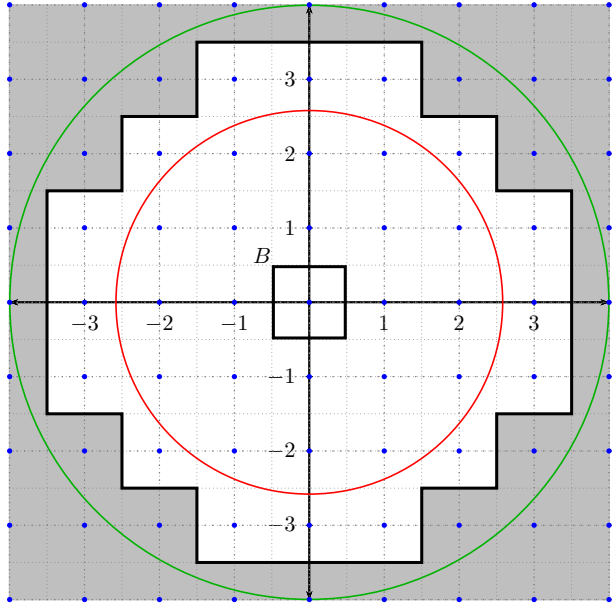}
		\caption{The sets $B$, $A'$, $\underline{A'}$ and $\overline{A'}$}
		\label{Fig2}
	\end{figure}	
	Applying $\phi$ to \eqref{eq:setinclusion} we get,
	\begin{equation*}
	\tub{\num{z}\leq \frac{1}{j+\frac{1}{\sqrt{2}}}}=\phi(\underline{A'})\subset A \subset \phi(\overline{A'})=\tub{\num{z}\leq \frac{1}{j-\frac{1}{\sqrt{2}}}}.
	\end{equation*}
	Taking the $\mu$-measure in the above inclusion followed by a small computation which makes use of the fact that $\mu$ is equivalent with respect to Lebesgue, we deduce that  
	\begin{equation}\label{3rdpower}
	\mu(A)=\mu\tub{\num{z}\leq \frac{1}{j}}+O\para{\frac{1}{j^3}}.
	\end{equation}
	We now proceed to estimate $\mu\tub{\num{z}\leq \frac{1}{j}}$, i.e. the $\mu$-measure of the ball centered at 0 with radius $\frac{1}{j}$. Set $D(j)=\tub{\num{z}\leq \frac{1}{j}}$ and
	\begin{equation*}
	D_i=D_i(j):=D(j)\cap A_i\enskip\text{for } i=1,\dots 8.
	\end{equation*}
	Note that the intersection with $A_9$, $A_{10}$, $A_{11}$, and $A_{12}$ is empty for $j$ sufficiently large, hence we may ignore these regions for our purposes. We have
	\begin{equation}
	\label{eq:1}
	\mu(D(j))=\mu\para{\bigcup_{i=1}^{8} D_i}=\sum_{i=1}^{8}\mu(D_i)=4\mu(D_1)+4\mu(D_5).
	\end{equation}
	The last equality follows from the symmetries noted in Theorem \ref{lem:1}. We can now use the fact that $\rho$ is continuous on each $D_i$ and that the limits
	\begin{equation*}
	\tilde{C}=\lim_{D_1\ni z\to 0}\rho(z)\quad\text{ and }\quad C'=\lim_{D_5\ni z\to 0}\rho(z).
	\end{equation*}
	exist and are strictly positive to estimate $\mu(D_1)$ and $\mu(D_5)$. Let $\epsilon>0$ be given. For $j$ sufficiently large we have  
	\begin{align*}
	\num{\rho(z)-\tilde{C}}&<\epsilon \enskip\text{if }z\in D_1(j),\\
	\num{\rho(z)-C'}&<\epsilon \enskip\text{if }z\in D_5(j)
	\end{align*}
	hence
	\begin{align*}
	(\tilde{C}-\epsilon)\lambda(D_1)\leq \mu(D_1)\leq (\tilde{C}+\epsilon)\lambda(D_1),\\
	(C'-\epsilon)\lambda(D_5)\leq \mu(D_5)\leq (C'+\epsilon)\lambda(D_5).
	\end{align*}
	We may now estimate $\mu(D)$ as follows
	\begin{equation}
	\label{muDjesti}
	4(\tilde{C}-\epsilon)\lambda(D_1)\leq\mu(D(j))\leq 4(\tilde{C}+\epsilon)\lambda(D_1)+4(C'+\epsilon)\lambda(D_5).
	\end{equation}
	The lower bound is motivated by Figure \ref{Fig4} which suggests that $\lambda(D_5)$ becomes insignificant compared to $\lambda(D_1)$ when $j$ becomes large. We make this precise by computing an upper bound for $\lambda(D_5)$.
	
	\begin{figure}[h]
		\centering
		\includegraphics[width=8cm]{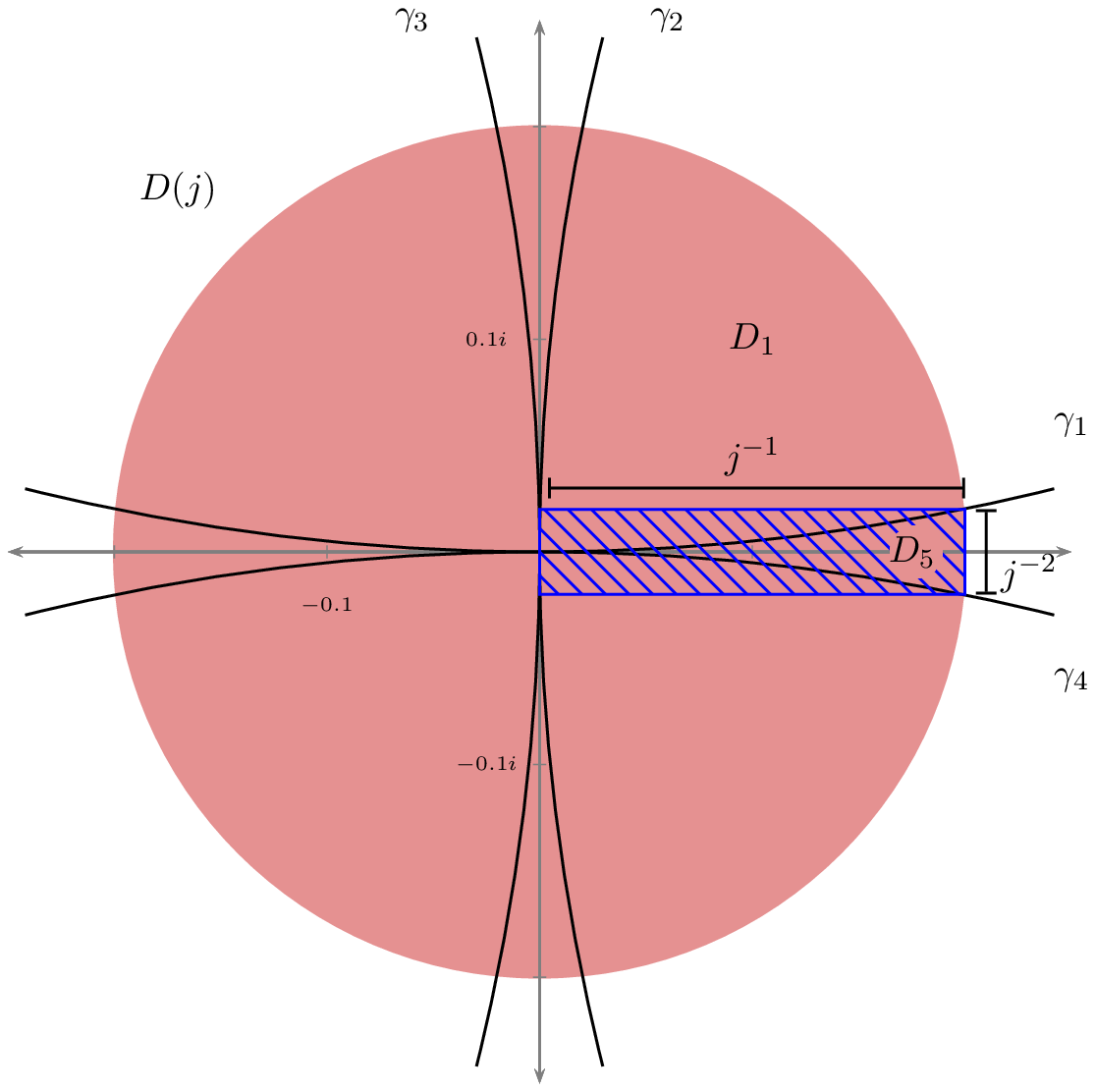}
		\caption{}
		\label{Fig4}
	\end{figure}
	
	We see from Figure \ref{Fig4} that $D_5$ is contained in a rectangle of side length $\frac{1}{j}$ and $\frac{1}{j^2}$. Here the vertical side length is found as the imaginary value of the intersection point between $\gamma_1$ and the boundary of $D$, i.e. solving
	\begin{equation*}
	x^2+(y-1)^2=1\enskip\text{and}\enskip x^2+y^2=\para{\frac{1}{j^2}},
	\end{equation*} 
	leading to $y=\frac{1}{2j^2}$ and side length $\frac{1}{j^2}$. Hence we get the bound
	\begin{equation*}
	\lambda(D_5)\leq \frac{1}{j}\frac{1}{j^2}=\frac{1}{j^3}.
	\end{equation*}
	We also use this bound to estimate $\lambda(D_1)$. Namely, we have $\lambda(D(j))=\pi\frac{1}{j^2}$ and
	\begin{equation*}
	\frac{\pi}{4}\frac{1}{j^2}=\frac14\lambda(D(j))\geq\lambda(D_1)=\frac14\lambda(D(j))-\lambda(D_5)\geq \frac{\pi}{4}\frac{1}{j^2}-\frac{1}{j^3}. 
	\end{equation*}
	Inserting these estimates in \eqref{muDjesti} we get
	\begin{equation*}
	(\tilde{C}-\epsilon)\pi\frac{1}{j^2}-4(\tilde{C}-\epsilon)\frac{1}{j^3}\leq\mu(D(j))\leq (\tilde{C}+\epsilon)\pi\frac{1}{j^2}+4(C'+\epsilon)\frac{1}{j^3}.
	\end{equation*}
	Since we are interested in the asymptotic behaviour as $j\to\infty$ we can pick $\epsilon$ arbitrarily small. Together with \eqref{3rdpower} we get the conclusion that for $H:=\pi\tilde{C}$ we have
	\begin{equation*}
	\mu\tub{z\in B:|a_1(z)| > j} =\frac{H}{j^2}+O\para{\frac{1}{j^3}},\quad \text{ as } j\to\infty.
	\end{equation*}
	This completes the proof of the lemma.
\end{proof}
\begin{remark}
	We repeat the point here that followed \eqref{eq:DensityForm}, namely that the constant $H$ could potentially be made explicit if one could find a way of computing the Lebegue measure of the sets $V_i$ which are defined and studied in \cite{EiItoNakadaNatsui} and \cite{HiaryVandehey}. 
\end{remark}

\begin{proof}[Proof of Theorem \ref{thm:main1}]
	As stated earlier, to finish the proof of Theorem \ref{thm:main1} we need to show that condition \eqref{DIG1} in Theorem \ref{IosifescuPoisson} is satisfied. However, at this stage we only need to observe that by choosing the sequence $u_n(r):=Cr\sqrt{n}$ with $C:=\sqrt{H}$ we get
	\begin{align*}
	\lim_{n\to\infty} n\mu\tub{\num{a_1}>Cr\sqrt{n}}&=\lim_{n\to\infty} n\para{\frac{H}{\para{Cr\sqrt{n}}^2}+O\para{\frac{1}{\para{Cr\sqrt{n}}^3}}}\\
	&=\frac{1}{r^2}+\lim_{n\to\infty} O\para{\frac{1}{\sqrt{n}r^3}}\\
	&=\frac{1}{r^2},
	\end{align*}
	and the proof is complete.
\end{proof}

\end{document}